\documentclass{amsproc}
\usepackage{amsmath}
\usepackage{enumerate}
\usepackage{amsthm}
\usepackage{mathtools}
\usepackage{amssymb}
\newtheorem{theorem}{Theorem}[section]
\newtheorem{lemma}[theorem]{Lemma}
\newtheorem{proposition}{Proposition}[section]
\newtheorem{corallary}{Corallary}[section]
\theoremstyle{definition}
\newtheorem{definition}[theorem]{Definition}

\theoremstyle{remark}

\numberwithin{equation}{section}

\begin{document}

\title{$\delta$-$J$-IDEALS OF COMMUTATIVE RINGS}

%    Remove any unused author tags.

%    author one information
\author{Shuai Zeng}
\address{ Shuai Zeng\\
School of Mathematics, Liaoning University, Shenyang, 110036, Liaoning Province, China }
%\curraddr{}
\email{shuaizseng@qq.com}
%\thanks{}

\author{Weiwei Wang}
\address{Weiwei Wang\\ School of Mathematics, Liaoning University, Shenyang, 110036, Liaoning Province, China }
%\curraddr{}
\email{wweiweimath@163.com}
%\thanks{}

%    author two information
\author{Jiantao Li }
\address{Jiantao Li (Corresponding author)\\
 School of Mathematics, Liaoning University, Shenyang, 110036, Liaoning Province, China }
%\curraddr{}
\email{jtlimath@qq.com}
\thanks{This work was supported by NSFC Grant 11701247.}



\date{}

\dedicatory{}

\begin{abstract}
	Let $\mathcal{I}(R)$ be the set of all ideals of a ring $R$, $\delta$ be an expansion function of $\mathcal{I}(R)$. In this paper, the $\delta$-$J$-ideal of a commutative ring is defined, that is, if $a, b\in R$ and $ab\in I\in \mathcal{I}(R)$, then $a\in J(R)$  (the Jacobson radical of $R$) or  $b\in \delta(I)$. Moreover, some properties of $\delta$-$J$-ideals are discussed,such as localizations, homomorphic images, idealization and so on.

\end{abstract}

\maketitle

\section{Introduction}

Throught this paper, let $R$ be a commutative ring with nonzero identity unless otherwise noted.

Since the prime ideals of commutative rings play a critical role in the area of commutative algebra,
various generalizaions of prime ideals have been explored in several studies.
 In \cite{DZ}, D. Zhao introduced the extension function $\delta$ of ideals and defined the $\delta$-primary ideals of rings.
Let $\mathcal{I}(R)$ be the set of all ideals of a ring $R$.
A function $\delta$ from $\mathcal{I}(R)$ to $\mathcal{I}(R)$  is said to be an expansion function or an ideal expansion, if it takes ideals to ideals with $I \subseteq \delta(I)$
and if $I\subseteq J$ for some $I, J\in \mathcal{I}(R) $, then $\delta(I)\subseteq\delta(J)$.
A $\delta$-primary ideal $I$ of a ring $R$ is an ideal satisfying that if $ab\in I$ and $a\notin I$ for some $a, b\in R$, then $b\in \delta(I)$.
In \cite{USK}, the authors introduced the concept of $n$-ideals: a proper ideal is called $n$-ideal if for any $a, b\in R$ with $ab\in I$, then $a\in \sqrt{0}$ or $b\in I=\delta_{0}(I)$.
For an ideal $I \in \mathcal{I}(R)$, $\sqrt{I}$ denots the radical of $I$, that is, $\sqrt{I}=\{r\in R| \text{ there exists } n\in \mathbb{N} \text{ such that } r^n\in I\}$.
Recently, E. Yetkin Celikel and G. Ulucak introduced the concepts of  $\delta$-$n$-ideals and quasi-$J$ ideals of a commutative ring $R$ in \cite{EYC,EG1}.
An ideal $I$ is said to be a $\delta$-$n$-ideal if for some $a, b\in R$ and $ab\in I$ with $a \notin \sqrt{0}$, then $b\in \delta(I)$.
For an ideal $I$, it is called a quasi-$J$-ideal of a ring if $ab\in I$ for some $a, b\in R$, then $a\in J(R)$ or $b\in \sqrt{I}=\delta_{1}(I)$. And this concept is a generalization of $J$-ideals appeared in \cite{KB}. 

In this paper, we defined the $\delta$-$J$-ideal of a commutative ring,  that is, if $a, b\in R$ and $ab\in I\in \mathcal{I}(R)$, then $a\in J(R)$ (the Jacobson radical of $R$)   or  $b\in \delta(I)$. This definition generalizes the results in \cite{EYC,EG1}.
More precisely, on the one side, $\sqrt{0}$ is always included in $J(R)$, thus this concept generalizes the  $\delta$-$n$-ideal in \cite{EG1}.
On the other hand, $\sqrt{I}=\delta_{1}(I)$ becomes a special expansion function of $\mathcal{I}(R)$, thus the concept of $\delta$-$J$-ideals generalizes the quasi-$J$-ideal in \cite{EYC}. In addtion, we also present some properties about the operations of $\delta$-$J$-ideals  such as localizations, homomorphic images, idealization and so on.

	\section{Properties of $\delta$-$J$-ideals}
Let $R$ be a commutative ring with nonzero identity $1$,  $\mathcal{I}(R)$ be the set of all ideals of $R$, $\delta$ be an expansion function of $\mathcal{I}(R)$, and $J(R)$ be the Jacobson radical of $R$.
For more definitions and properties of idealization used in this paper, one can refer to literatures \cite{MF,DDM,JH}.

Now, let's present the definition of the $\delta$-$J$-ideal.
\begin{definition}
	Let $R$ be a ring, a proper ideal $I$ of $R$ is said to be a $\delta$-$J$-ideal, if whenever $a, b \in R$, $ab \in I$,  then $a\in J(R)$  or $b\in \delta(I)$. 
\end{definition}

This definition is different to the conccepts of $\delta$-$n$-ideals and quasi-$J$-ideals, and it is more general than these two concepts.

In the following result, we presented some equivalent conditions of $\delta$-$J$-ideal of a ring in some special circumstance.
\begin{theorem}\label{th1}
	Let $I$ be a proper ideal of ring $R$. Then the following statements are equivalent:
	\begin{enumerate}[(1)]
		\item$I$ is a $\delta$-$J$-ideal of $R$.
		\item If $a \in R$ and $K$ is an ideal of $R$ with $aK$$\subseteq$$I$, then $a\in J(R)$ or $K\subseteq \delta (I)$ if $a\notin J(R)$.
		\item If $K$ and $L$ are ideals of $R$ with $KL \subseteq I$, then $K \subseteq J(R)$ or $L \subseteq \delta(I)$.
	\end{enumerate}
\end{theorem}

\begin{proof}
	(1)$\Rightarrow$(2): Suppose that $I$ is a $\delta$-$J$-ideal of $R$, $aK\subseteq I$ and $a \notin J(R)$. Then if $as \in I$, for any $s\in K$, it follows that $s\in \delta(I) $.  Furthermore, $K\subseteq \delta(I)$.
	
	(2)$\Rightarrow$(3): Suppose that $KL\subseteq I$ and $K\nsubseteq J(R)$. Then there exists $a \in K\setminus J(R)$. Since $aL \subseteq I$ and $a \notin J(R)$, we have $L\subseteq \delta(I)$ by $(2)$.
	
	(3)$\Rightarrow$(1): Suppose that $a, b \in R$ and $ab \in I$, the result follows by letting $K= \langle a \rangle $ and $L= \langle b \rangle $ in $(3)$.
\end{proof}

\begin{proposition}\label{prop1}
	Let $I$ be an ideal of $R$ with $\delta(I)\neq R$, then the following statements are equivalent:
	\begin{enumerate}[(1)]
		\item $I$ is a $\delta$-$J$-ideal of $R$.
		\item $I\subseteq J(R)$ and if whenever $a, b \in R$ with $ab \in I$, the $a\in J(I)$ or $b\in \delta(I)$.
	\end{enumerate}
\end{proposition}

\begin{proof}
	(1)$\Rightarrow$(2): Suppose that $I$ be a $\delta$-$J$-ideal of $R$. Assume that $I\nsubseteq J(R)$. 
Then there exists an element $a\in R$ with $a\in I\setminus J(R)$. Since $a= a\cdot 1\in I$ and $a\notin J(R)$, it follows that $1\in \delta(I)$, which is a contradition. 
Thus $I\subseteq J(R)$, and the other part in $(2)$ follows clearly since $J(R)\subseteq J(I)$.
	
	(2)$\Rightarrow$(1): Suppose that $ab \in I$ and $a\notin J(R)$. Since $I \subseteq J(R)$, we conclude that $J(I)\subseteq J(J(R))= J(R)$ and $a \notin J(I)$. 
Thus, $b\in \delta(I)$,  $I$ is a $\delta$-$J$-ideal of $R$.
	
\end{proof}
Next, we characterize the ring which every proper ideal is a $\delta$-$J$-ideal.

\begin{theorem}
	Let $R$ be a ring such that $\delta(I)\neq R$ for all proper ideal $I$ of  $R$. Then the following statements are equivalent:
	\begin{enumerate}[(1)]
		\item $R$ is a quasi-local ring with maximal ideal $M= J(R)$.
		\item Every proper principal ideal is a $\delta$-$J$-ideal of $R$.
		\item Every proper ideal is a $\delta$-$J$-ideal of $R$.
	\end{enumerate}
\end{theorem}

\begin{proof}
	(1)$\Rightarrow$(2): Suppose that $(R, J(R))$ is a quasi-local ring. 
Then every element of $R$ is either an unit or belongs to $J(R)$. 
Let $I=\langle x\rangle$ be a principal ideal and let $a, b\in R$ such that $ab\in I$ and $a\notin J(R)$. Then $a$ is an unit and $b\in I\subseteq \delta(I)$. So $I$ is a $\delta$-$J$-ideal.
	
	(2)$\Rightarrow$(3): Let $I$ be a proper ideal of $R$. Suppose that $a,b \in  R$ with $ab\in I$ and $a\notin J(R)$. Put $J=\langle ab\rangle$.Then $J$ is a $\delta$-$J$-ideal by (2). Therefore, $b\in \delta(J)\subseteq \delta(I)$, $I$ is also a $\delta$-$J$-ideal.
	
	(3)$\Rightarrow$(1): Let $I$ be a maximal ideal of $R$. Then $I$ is a $\delta$-$J$-ideal by assumption, thus $I=J(R)$ by Proposition \ref{prop1}.
\end{proof}
The following lemma gives a relationship between $\delta$-primary ideal and $\delta$-$J$-ideal.  We also consider the connection between maximal ideals and $\delta$-$J$-ideals.

\begin{lemma}
	Let $\delta$ be an expansion of ideals of $R$. Then the followings are hold:
	
	\begin{enumerate}[(1)]
		\item Let $I$ be a $\delta$-primary ideal of $R$ with $\delta(I)\neq R$. Then $I$ is a $\delta$-$J$-ideal of $R$ if and only if $I\subseteq J(R)$.
		\item Let $I$ be a maximal ideal of $R$ with $\delta(I)\neq R$. Then $I$ is a $\delta$-$J$-ideal of $R$ if and only if $I=J(R)$.
		
	\end{enumerate}
\end{lemma}

\begin{proof}
	(1) Suppose that $I\subseteq J(R)$. Choose $a, b \in R$ with $ab\in I$ and $a\notin J(R)$, hence $a\notin I$. 
Since $I$ is a $\delta$-primary ideal, we have $b\in \delta(I)$, that is,  $I$ is a $\delta$-$J$-ideal. The other side is obvious.

	(2) Suppose that $I$ is a $\delta$-$J$-ideal. Since $I$ is maximal, we have $J(R)\subseteq I$. Then $I$ is a $\delta$-primary ideal. It follows from (1) that $I\subseteq J(R)$. Thus the equality $I=J(R)$ holds. Conversely, if $I=J(R)$, then $I$ is a $\delta$-$J$-ideal.
	\end{proof}

Now, we recall the concept of ideal quotient. Let $I, J$ be ideals of $R$, their ideal quotient is $(I:J)=\{x\in R|xJ\subseteq I\}$. If $J$ is a principal ideal $\langle x \rangle$, we shall write $(I:x)$ in place of $(I:\langle x \rangle)$.

\begin{lemma}\label{lemma2}
	Let $\delta$ be an expansion of $\mathcal{I}(R)$. If $I$ is a $\delta $-$J$-ideal of $R$ such that $(\delta(I): x) \subseteq \delta ((I: x))\neq R$ for all $x \notin R\setminus \delta(I)$, then $(I: x)$ is a $\delta$-$J$-ideal.
\end{lemma}

\begin{proof}
	Suppose that $a, b\in R$ with $ab\in (I:x)$ and $a\notin J(R)$. Since $abx \in I$ and $I$ is $\delta$-$J$-ideal, we have $bx \in \delta(I)$. Thus $b\in (\delta(I): x)\subseteq \delta((I: x))$, $(I: x)$ is a $\delta$-$J$-ideal.
\end{proof}

\begin{theorem}\label{th6}
	Let $I$ be a maximal $\delta$-$J$-ideal of $R$ and $(\delta(I): x)\subseteq \delta((I: x))$, for all $x\notin R\setminus(\delta(I)\cup J(R))$. Then $I$ is a $J$-ideal.
	
\end{theorem}

\begin{proof}
	Suppose that $I$ is a maximal $\delta$-$J$-ideal of $R$. Let $a, b\in R$ such that $ab\in I$ and $a\notin J(R)$. Then $(I: a)$ is a $\delta$-$J$-ideal of $R$ by Lemma \ref{lemma2}. Since $I$ is a maximal $\delta$-$J$-ideal and $I\subseteq (I: a)$, then $b\in (I: a)= I$. Therefore, $I$ is a $J$-ideal of $R$.
\end{proof}	

Now, we can propose the following results.
\begin{corallary}
	Let $R$ be a ring, $\delta(I)\neq R$ for all proper ideal of $R$, and $(\delta(I): x)\subseteq \delta((I: x))$ for all $x\notin R\setminus(\delta(I)\cup J(R))$. Then the following statements are equivalent:
	\begin{enumerate}[(1)]
		\item $J(R)$ is a $J$-ideal of $R$.
		\item $J(R)$ is a $\delta$-$J$-ideal  of $R$.
		\item $J(R)$ is a prime ideal of $R$.
	\end{enumerate}
\end{corallary}

\begin{proof}
	(1)$\Leftrightarrow$(3): Clearly.
	
	(1)$\Rightarrow$(2): Clearly.
	
	(2)$\Rightarrow$(1): For any $\delta$-$J$-ideal of $R$, we have $I\subseteq J(R)$. Thus $J(R)$ is the maximal $\delta$-$J$-ideal of $R$. Then $J(R)$ is a $J$-ideal by Theorem $\ref{th6}$.
\end{proof}

\begin{proposition}
	Let $\delta$ be an expansion function of  $\mathcal{I}(R)$ and $I$ be a proper ideal of $R$ with $\delta(\delta(I))= \delta(I)$. Then the followings are holds:
	
	\begin{enumerate}[(1)]
		\item If $I$ is a $\delta$-$J$-ideal and $a\notin J(R)$, then $\delta((I: a))= \delta(I)$.
		\item $\delta(I)$ is a $J$-ideal if and only if $\delta(I)$ is a $\delta$-$J$-ideal.
		\item If $IK = JK $ and $I, J$ are $\delta$-$J$-ideals of $R$ with $\delta(\delta(J))= \delta(J)$ and $K \cap (R-J(R)) \neq \emptyset$ for some ideal $K$ of $R$, then $\delta(I)= \delta(J)$.
		\item If $IK$ and $I$ are $\delta$-$J$-ideals of $R$ with $\delta(\delta(IK))= \delta(IK)$ and $K\cap (R-J(R))\neq \emptyset$ for some ideal $K$ of $R$, then $\delta(IK)= \delta(I)$.
	\end{enumerate}
\end{proposition}

\begin{proof}
	(1) Let $I$ be a $\delta$-$J$-ideal and $a\notin J(R)$. Note that $I\subseteq (I: a)$, so $\delta(I)\subseteq \delta((I: a))$.
 Let $x\in (I: a)$. Then $x\in \delta(I)$. Since $xa\in I$ and $a\notin J(R)$. Thus $(I: a)\subseteq \delta(I)$, we get $\delta(I: a)\subseteq \delta(\delta(I))=\delta(I)$. The equality follows.
	
	(2) Clearly.
	
	(3) Since $IK = JK\subseteq I, J$. It follows from $JK \subseteq I$ and $K\cap (R-J(R))\neq \emptyset$ that $J\subseteq \delta(I)$. 
Similarly, $I\subseteq \delta(J)$.  Moreover, $\delta(\delta(I))=\delta(I)$ and $\delta(\delta(J))=\delta(J)$, Thus $\delta(I)=\delta(J)$.
	
	(4) Since $IK\subseteq I$, we have $\delta(IK)\subseteq \delta(I)$. We can also get $I\subseteq \delta(IK)$ since $IK\subseteq IK$.
 Note that $K \cap (R-J(R))\neq \emptyset$, thus $\delta(IK)=\delta(I)$ by assumption.
	
\end{proof}

\begin{proposition}
	Let $\delta $ and $\gamma$ be expansion functions of $R$ and $I$ be a proper ideal of $R$. Then:
	
	\begin{enumerate}[(1)]
		\item If $\delta(I)$ is a $J$-ideal of $R$, then $I$ is a $\delta$-$J$-ideal of $R$. The converse is also true if $\delta=\delta_{1}$.
		\item Suppose that $\delta(I)\subseteq \gamma(I)$ for any ideal $I$ of $R$. If $I$ is a $\delta$-$J$-ideal of $R$, then $I$ is a $\gamma$-$J$-ideal of $R$.
		\item If $\gamma(I)$ is a $\delta$-$J$-ideal of $R$, then $I$ is a $\delta\circ \gamma$-$J$-ideal of $R$.
	\end{enumerate}
\end{proposition}

\begin{proof}
	(1) Suppose that $ab\in I$ and $a\notin J(R)$ for some $a, b \in R$. Since $I\subseteq \delta(I)$ and $\delta(I)$ is a $J$-ideal, we get $b\in \delta(I)$. Thus $I$ is a $\delta$-$J$-ideal of $R$. Conversely, suppose that $ab\in \delta_{1}(I)$ with $a\notin J(R)$ for some $a, b\in R$, then $(ab)^{n}= a^{n}b^{n}\in I$ for some $n\geq 1$, and clearly $a^{n}\notin J(R)$ because $\delta_{1}(J(R))=J(R)$ (see \cite{MF}). Since $I$ is a $\delta_{1}$-$J$-ideal, we have $b^n\in \delta_{1}(I)$. Thus $b\in \delta_{1}(I)$, as required.
	
	(2) Clearly.
	
	(3) Assume that $\gamma(I)$ is a $\delta$-$J$-ideal of $R$. Let $ab\in I$ for some $a, b \in R$ and $a\notin J(R)$. Since $I \subseteq\gamma(I)$, we have $ab\in \gamma(I)$. Since $\gamma(I)$ is a $\delta$-$J$-ideal of $R$, $b\in \delta(\gamma(I))= \delta \circ \gamma(I)$. 	
\end{proof}

\begin{proposition}
	Let $I, J$ and $K$ be proper ideals of $R$ with $J\subseteq K\subseteq I$. If $I$ is a $\delta$-$J$-ideal of $R$, and $\delta(J)= \delta(I)$, then $K$ is a $\delta$-$J$-ideal of $R$.
\end{proposition}

\begin{proof}
Assume that $I$ is a $\delta$-$J$-ideal of $R$ and $\delta(J)= \delta(I)$. Let $ab\in K$ for some $a, b\in R$. Then by Theorem \ref{th1}, $a\in J(R)$ or $b\in \delta(K)$ since $K\subseteq I$ and $\delta(J) = \delta(I)= \delta(K)$. Thus, $K$ is a $\delta$-$J$-ideal of $R$.
\end{proof}

An expansion $\delta$ is said to be a intersection preserving function if $\delta(I)\cap\delta(J)= \delta(I\cap J)$ for any $I, J\in \mathcal{I}(R)$, see \cite{DZ} for more details.
\begin{proposition}
	Let $\delta$ be an ideal expansion which preserves intersection. Then the following statements are hold:
	
	\begin{enumerate}[(1)]
		\item If $I_{1}, I_{2}, \ldots, I_{n}$ are $\delta$-$J$-ideals of $R$, then $I= \bigcap_{i=1}^{n} I_{i}$ is a $\delta$-$J$-ideal of $R$.
		\item Let $I_{1}, I_{2}, \ldots, I_{n}$ be of $R$ such that $\delta(I_{i})'s$ are non-comparable prime ideals of each other of $R$. If $I= \bigcap_{i=1}^{n} I_{i}$ is a $\delta$-$J$-ideal of $R$, then $I_{i}$ is a $\delta$-$J$-ideal of $R$ for all $i = 1, 2, \cdots, n$.
	\end{enumerate}
\end{proposition}

\begin{proof}
	(1) Let $ab\in I$ and $b\notin \delta(I)$ for some $a, b \in R$. Since $\delta(I)= \bigcap_{i=1}^{n}\delta(I_{i})$, we have $b\notin \delta(I_{k})$ for some $k\in \{1, \ldots , n\}$. It follows that $a\in J(R)$. Thus $I$ is a $\delta$-$J$-ideal.
	
	(2) Suppose that $ab \in I_{k}$ and $a\notin J(R)$ for some $k\in \left\{1, \ldots , n\right\}$. Choose $x\in \left(\prod_{i=1,i\neq k}^{n}I_{i} \right)\setminus\delta(I_{k})$. Hence $abx\in \bigcap_{i=1}^{n}I_{i}$. Since $ \bigcap_{i=1}^{n}I_{i}$ is a $\delta$-$J$-ideal, we have $bx\in \delta( \bigcap_{i=1}^{n}I_{i})= \bigcap_{i=1}^{n}\delta(I_{i})\subseteq \delta(I_{k})$. Note that $\delta(I_{k})$ is prime, thus $b\in \delta(I_{k})$.	
\end{proof}

Let $R$ and $S$ be two commutative rings and $\delta, \gamma$ be expansion functions of $\mathcal{I}(R), \mathcal{I}(S)$, respectively. If $\delta(f^{-1}(J))= f^{-1}(\gamma(J))$ for all ideal $J$ of $S$ with ring homomorphism $f : R\rightarrow S$, then we call $f$ is a $\delta\gamma$-homomorphism.

\begin{proposition}
	Let $f : R\rightarrow S$ be a  $\delta\gamma$-homomorphism, wher $\delta$ and $\gamma$ are expansion functions of $\mathcal{I}(R)$ and $\mathcal{I}(S)$, respectively. Then the following hold:
	\begin{enumerate}[(1)]
		\item Let $f$ be a monomorphism. If $J$ is a $\gamma$-$J$-ideal of $S$. Then $f^{-1}(J)$ is a $\delta$-$J$-ideal of $R$.
		\item Suppose that $f$ is an epimorphism and $I$ is a proper ideal of $R$ with $\ker f\subseteq I$. If $I$ is a $\delta$-$J$-ideal of $R$, then $f(I)$ is a $\gamma$-$J$-ideal of $S$.
	\end{enumerate}
\end{proposition}

\begin{proof}
	(1) Let $ab\in f^{-1}(J)$ for $a, b\in R$. Then $f(ab)= f(a)f(b)\in J$, which implies $f(a)\in J(S)$ or $f(b)\in \gamma(J)$. 
If $f(a)\in J(S)$, then $a\in J(R)$ since $f$ is a monomorphism. 
If $f(b)\in \gamma(J)$, then we have $b\in f^{-1}(\gamma(J))= \delta(f^{-1}(J))$ since $f$ is a $\delta\gamma$-homomorphism.
 Thus $f^{-1}$ is a $\delta$-$J$-ideal of $R$.
	
	(2) Suppose that $a, b\in S$ with $ab\in f(I)$ and $a\notin J(S)$. Since $f$ is an epimorphism, there exist $x, y \in R$ such that $a= f(x), b=f(y)$. 
Thus we have $x\notin J(R)$ from $a\notin J(S)$. Since $\ker f \subseteq I$, $ab= f(xy)\in f(I)$ implies that $xy\in I$. Thus $y\in \delta(I)$ and $b=f(y)\in f(\delta(I))$. On the other hand $\gamma(f(I))=f(\delta(I))$, we have $b\in \gamma(f(I))$. Thus $f(I)$ is a $\gamma$-$J$-ideal of $S$.
	
\end{proof}

Let $\delta$ be an expansion function of $\mathcal{I}(R)$, $I$ be an ideal of $R$. Then the function $\delta_{q} : R\slash I \rightarrow R\slash I$ defined by $\delta_{q}(J\slash I)=\delta(J)\slash I$ for all ideals $I\subseteq J$, becomes an expansion function of $\mathcal{I}(R\slash I)$.

\begin{corallary}
	Let $\delta$ be an expansion function of $\mathcal{I}(R)$ and $J\subseteq I$ be proper ideals of $R$. Then the followings hold:
	\begin{enumerate}[(1)]
		\item If $I$ is a $\delta$-$J$-ideal of $R$, then $I\slash J$ is a $\delta_{q}$-$J$-ideal of $R\slash J$.
		\item If $I\slash J$ is a $\delta_{q}$-$J$-ideal of $R\slash J$ and $J\subseteq J(R)$, then $I$ is a $\delta$-$J$-ideal of $R$.
		\item If $I\slash J$ is a $\delta_{q}$-$J$-ideal of $R\slash J$, and $J$ is a $\delta$-$J$-ideal of $R$ where $\delta(J)\neq R$, then $I$ is a $\delta$-$J$-ideal of $R$.
		\item Let $S$ be a subring of $R$ with $S\nsubseteq I$. Then $S\cap I$ is a $\delta$-$J$-ideal of $R$.
		
	\end{enumerate}
\end{corallary}

\begin{proof}
	(1) Take the natural homomorphism $\pi : R\rightarrow R\slash J$. Then we have that $I\slash J$ is a $\delta_{q}$-$J$-ideal of $R\slash J$ since $\ker \pi\subseteq I$.
	
	(2) Let $I\slash J$ be a $\delta_{q}$-$J$-ideal of $ R\slash J$ and $J\subseteq J(R)$. Assume that $ab\in I$ and $a\notin J(R)$ for some $a, b\in R$. Then $ab+J=(a+J)(b+J)\in I\slash J$ and $a+J\notin J(R\slash J)$ by assumption, $b+J\in \delta_{q}(I\slash J)= \delta(I)\slash J$, that is $b\in \delta(I)$.
	
	(3) Clearly.
	
	(4) The result follows just from taking the injection $i : S\rightarrow R$ defined by $i(a)=a$ for every $a\in S$. 
\end{proof}

A proper ideal $I$ of a ring $R$ is called a superfluous ideal of $R$ if there is no proper ideal $J$ of $R$ such that $I + J= R$.

\begin{lemma}\label{lemma2.7}
	Any $\delta$-$J$-ideal of a ring $R$ with $\delta(I)\neq R$ is superfluous.
\end{lemma}	
\begin{proof}
	Let $I$ be a $\delta$-$J$-ideal fo $R$ with $\delta(I)\neq R$. Assume that there exists a proper ideal $J$ of $R$ with $I+ J =   R$. Then $1= a + b$ for some $a\in I, b\in J$ and so $1-b\in I\subseteq J(R)$. Thus $b\in J$ is an unit and we get $J= R$, a contradiction.	
\end{proof}

\begin{proposition}
	Let $I$ and $J$ be $\delta$-$J$-ideals of $R$ such that $\delta(I)\neq R$, $\delta(J)\neq R$. Then $I+ J $ is a $\delta$-$J$-ideal of $R$.
\end{proposition}	

\begin{proof}
	Let $I,  J$ be $\delta$-$J$-ideals of $R$ such that $\delta(I)\neq R$, $\delta(J)\neq R$. Since $I, J$ are superfluous by Lemma \ref{lemma2.7}, $I + J \neq R$. Hence $I\cap J$ is a $\delta$-$J$-ideal. Also, $I\slash (I\cap J)$ is a $\delta_{q}$-$J$-ideal of $R\slash (I\cap J)$. Now, by the isomorphism $I\slash(I\cap J) \cong (I  + J)\slash J$, we have $(I  + J)\slash J$ is a $\delta_{q}$-$J$-ideal of $R\slash J$, therefore $I+ J$ is a $\delta$-$J$-ideal of $R$.
\end{proof}

Let $S$ be a multiplicatively closed subset of $R$. Denote by $S^{-1}R$ the ring of fractions of $R$ with respect to $S$.
Note that $\delta_{S}$ is an expansion function of $\mathcal{I}(S^{-1}R)$ such that  $\delta_{S}(S^{-1}I)= S^{-1}(\delta(I))$, where $\delta$ is an expansion function of $R$. Denote $Z_{I}(R)=\{r\in R | rs \in I  \text{ for some } s \in R\setminus I\}$, where $I$ is a proper ideal of $R$. 

\begin{proposition}
	Let $S$ be a multiplicatively closed subset of $R$ and $\delta$ be an expansion function of $R$ and $J(S^{-1}R)= S^{-1}J(R)$. Then 	
	\begin{enumerate}[(1)]
		\item If $I$ is a $\delta$-$J$-ideal of $R$ with $I\cap S  = \emptyset$, then $S^{-1}I $ is a $\delta_{S}$-$J$-ideal of $S^{-1}R$.
		\item  Let $S\cap Z_{J(R)}(R)= S\cap Z_{\delta(I)}(R)= \emptyset$. If $S^{-1}I $ is a $\delta_{S}$-$J$-ideal of $S^{-1}R$, then $I$ is a $\delta$-$J$-ideal of $R$.
	\end{enumerate}
\end{proposition}

\begin{proof}
	(1) Suppose $\frac{a}{s}\cdot \frac{b}{t} \in S^{-1}I$ and $\frac{a}{s}\notin J(S^{-1}R)$ for some $a, b \in R$ and $s, t \in S$. Then there exists $u\in S$ with $abu\in I$, thus $bu\in \delta(I)$ since $a\notin J(R)$. Hence $\frac{b}{t} = \frac{bu}{tu}\in S^{-1}(\delta(I))= \delta_{S}(S^{-1}I)$. Consequently, $S^{-1}I$ is a $\delta_{S}$-$J$-ideal of $S^{-1}R$.
	
	(2) Let $a, b \in R$ with $ab\in I$. Then $\frac{a}{1}\cdot \frac{b}{1}\in S^{-1}I$, It implies that either $\frac{a}{1}\in J(S^{-1}R)= S^{-1}J(R)$ or $\frac{b}{1}\in \delta_{S}(S^{-1}I)=S^{-1}\delta(I)$. So if $\frac{a}{1}\in S^{-1}J(R)$, then we can deduce that $as \in J(R)$ for some $s\in S$. It follows from $S\cap Z_{J(R)}(R)= \emptyset$ that $a\in J(R)$.
 If $\frac{b}{1}\in S^{-1}\delta(I)$, by a similar disscution, we can get $b\in \delta(I)$. 
\end{proof}

Let $R(+)M$ be the idealization of an $R$-module $M$. For an expansion function $\delta$ of $R$, define $\delta_{(+)}$ as $\delta_{(+)}(I (+) N)= \delta(I)(+)M$ for some ideal $I(+)N$ of $R(+)M$. It is clear that $\delta_{(+)}$ is an expansion function of $R(+)M$.

\begin{proposition}
	Let $I$ be an ideal of a ring $R$ and $N$ be a submodule of an $R$-module $M$. Then $I$ is a $\delta$-$J$-ideal of $R$ if and only if $I(+)N$ is a $\delta_{(+)}$-$J$-ideal of $R(+)M$.
\end{proposition}	

\begin{proof}
	Let $I$ be a $\delta$-$J$-ideal of $R$. Assume that $(r, m), (s, m')\in I(+)N$ and $(s, m')\notin J(R)(+)M$ for some $(r, m)(s, m')\in R(+)M$. Then $s\in \delta(I)$ since $rs\in I$ and $s\notin J(R)$. Thus $(s, m')\in \delta(I)(+)M= \delta_{(+)}(I(+)N)$. 

Conversely, suppose that $I(+)N$ is a $\delta_{(+)}$-$J$-ideal of $R(+)M$. Let $r, s \in R$ with $rs\in I$ and $s\notin J(R)$. Hence, we get $(r, m)(s, m')\in I(+)N$ and clearly $(s, m')\notin J(R)(+)M$. Therefore, $(r, m)\in \delta_{(+)}(I(+)N)$, and $r\in \delta(I)$, $I$ is a $\delta$-$J$-ideal.
\end{proof}


\begin{thebibliography}{99}
        \bibitem{DZ} D. Zhao, $\delta$-primary ideals of commutative rings, Kyungpook Mathematical Journal,41 (2001), 17--22.
		\bibitem{USK}   U. Tekir, S. Koc and K. H. Oral, $n$-ideals of commutative rings, Filomat, 31 (10) (2017),
	2933--2941.
	   \bibitem{EYC} E. Yetkin Celikel, G. Ulucak, $\delta$-$n$-ideals of a commutative ring. arXiv: 2103.11679v1.
       \bibitem{EG1} H. A. Khashan, E. Yetkin Celikel, Quasi-$J$-ideals of commutative rings. arXiv: 2102.10299v1.
	    \bibitem{KB} H. A. Khashan, A. B. Bani-Ata,J-ideals of commutative rings. International Electronic Journal of Algebra, 2021, 29(1):148-164.
	   \bibitem{MF} M. F. Atiyah, I. G. MacDonald, Introduction to commutative algebra,
	    Addison-Wesley Longman, 1969.
	    \bibitem{DDM} D. D. Anderson, M. Winders, Idealization of a Module, Journal of Commutative Algebra, 1
	    (1) (2019), 3-56.
	    \bibitem{JH} J. Huckaba, Rings wiht zero-divisors, Marcel Dekker, NewYork/Basil, 1998.
	   % \bibitem{DDS} %D. D. Anderson, and S. Valdes-Leon, Factorization in commutative rings with with zero
	   % divisors. Rocky Mountain J. Math, 26 (1996), 439 - 480.
	   
\end{thebibliography}
\end{document}